\documentclass[11pt,reqno]{amsart}
\def\@rmrk#1#2{\refstepcounter
    {#1}\@ifnextchar[{\@yrmrk{#1}{#2}}{\@xrmrk{#1}{#2}}}

\newfont{\bfit}{cmbxti10 scaled 2000}
 \newfont{\biggi}{cmr12 scaled 2000}

 \newcommand{\skrie}{{\mathcal E}}
 
 \newcommand{\skrig}{{\mathcal G}}

 \newcommand{\skrik}{{\mathcal K}}
 \newcommand{\skril}{{\mathcal L}}

 \newcommand{\skriv}{{\mathcal V}}
 
 \newcommand{\skriy}{{\mathcal Y}}
 
 \newcommand{\heap}[2]{\genfrac{}{}{0pt}{}{#1}{#2}}
 \newcommand{\sfrac}[2]{\mbox{$\frac{#1}{#2}$}}

\def\1{{\mathchoice {1\mskip-4mu\mathrm l}      
{1\mskip-4mu\mathrm l}
{1\mskip-4.5mu\mathrm l} {1\mskip-5mu\mathrm l}}}

\newcommand{\eq}{\begin{equation}}
\newcommand{\en}{\end{equation}}

{\nopagebreak {\hspace*{\fill}\rule{2mm}{2mm}}\\ }
{\nopagebreak {\hspace*{\fill}\rule{2mm}{2mm}}\\ }

\renewcommand{\subsection}{\secdef \subsct\sbsect}
\newcommand{\subsct}[2][default]{\refstepcounter{subsection}
\vspace{0.15cm}
{\flushleft\bf \arabic{section}.\arabic{subsection}~\bf #1  }
\nopagebreak\nopagebreak}
\newcommand{\sbsect}[1]{\vspace{0.1cm}\noindent
{\bf #1}\vspace{0.1cm}}

\newtheorem{theorem}{Theorem}[section]
\newtheorem{lemma}[theorem]{Lemma}

\newtheoremstyle{thm}{1.5ex}{1.5ex}{\itshape\rmfamily}{}
{\bfseries\rmfamily}{}{2ex}{}

\newtheoremstyle{rem}{1.3ex}{1.3ex}{\rmfamily}{}
{\itshape\rmfamily}{}{1.5ex}{}
\theoremstyle{rem}

\refstepcounter{subsection}

\def\thebibliography#1{\section*{References}
  \list%
  {\arabic{enumi}.}
    {\settowidth\labelwidth{[#1]}\leftmargin\labelwidth
    \advance\leftmargin\labelsep
    \parsep0pt\itemsep0pt
    \usecounter{enumi}}
    \def\newblock{\hskip .11em plus .33em minus .07em}
    \sloppy                   
    \sfcode`\.=1000\relax}

\begin{document}
\title[Joint  Large  Deviation  principle   for  empirical  measures of the   d-regular  random  graphs] {Joint Large  Deviation  principle for  empirical  measures of the d-regular  random  graphs}
\author[K.D-A]{}

\maketitle
\thispagestyle{empty}
\vspace{-0.5cm}

\centerline{\sc{U. Ibrahim, A. Lotsi  and  K. Doku-Amponsah$^{1}$}}
\centerline{\sc{University  of  Ghana}}
\centerline{College  of  Basic  and  Applied  Sciences, School  of  Physical  and  Mathematical  Sciences}
\centerline{Statistics  Department, Box LG 115, Legon-Accra, West Africa}
\centerline{Email:$kdoku-amponsah@ug.edu.gh^1$}

\vspace{0.5cm}

 \centerline{\bf \small Abstract}
\begin{quote}{\small For a $d-$regular  random model,  we  assign to  vertices $q-$state spins. From this  model,  we define   the \emph{empirical co-operate measure}, which  enumerates the  number  of  co-operation between a  given couple of spins,  and \emph{ empirical spin measure},   which  enumerates the  number of sites having a given spin  on the  $d-$regular random graph  model.  For  these  empirical measures we  obtain  large  deviation  principle(LDP) in  the  weak  topology.} 
\end{quote} \vspace{0.5cm}

\textit{Keywords: } $d-$regular  random graph, random partition function, empirical  cooperate  measure, empirical spin  measure,large deviation principle.

{\textit{Mathematics Subject Classification :}   60F10,05C80.\\

\section{Introduction}
The structure of a network (graph) affects its performance. For example, the topology of a social network influences the spread of information or diseases ( Strogatz, 2001).Moreover, it is not easy to study the properties of a uniform regular random graph if cannot be generated. The cardinality of the graph can only be found if the graph can be generated. In fact it a strenuous task to try to get a finite list of the graph because the number of regular graphs is very gargantuan. A superior method to follow is the Configuration Model proposed by Bollabàs (1980) which was inspired in parts by Bender and Canfield( 1978). The configuration or pairing model is used to generate firstly uniformly at random a set of multigraphs (with loops and multi-edges ) and conditioning on its distribution having no loops or multi-edges, a simple regular  graph is generated.\\

We  consider  in  this  paper   uniform d-regular  random  graph  model constructed  via  pairing model,  See Bollobas (1980),  as  follows:  Let   $D=(d_1,d_2,d_2,...,d_n)$  where  $d_i=d$  for  all  $i= 1,2,3,...,n$  and  $nd$  even. For  each  vertex $i\in[n]$ we represent  as  a  bin $v_i$  and  we place $d_i$  distinct points  in bin  $v_i$  for  every  $1\le i\le n.$  Take  a  uniformly random  perfect matching $\sum_{i=1}^{n}d_i=nd$ points. This  perfect  matching  is  is  called a  pairing; each  pair  of  points joined  in  the  matching  is  called  a  pair  of  the pairing.  Each  pairing  corresponds  to  multi-graph  and  conditioned on  the set  of  simple  graphs  we  obtain  the  uniformly  $d-$  regular  random  graph.\\

In this paper we prove a  joint  LDP  for  the  empirical  spin  measure   and  empirical  co-operate  measure  of  the  $q$-  state  $d-$ regular  random  graph. Specifically,  we  prove  a joint  LDP  for  the  empirical  spin  measure and  empirical  co-operate  measure   via  combinatoric  argument, see  Dembo et al.~(2014, Page~10)   and  the  method  of  types, see  Dembo and Zeitouni ~(1998, Pages~16-17).

The  remaining  part  of  the  article is  organized as follows: We present in Subsection~\ref{model} the  q-State $d-$regular random graph. The main result of  this  article, Theorem~\ref{mainpart}  is presented in Subsection~\ref{main}.Subsection~\ref{LDP} contains the  LDP probabilities, Lemma~\ref{LDPP1}   and  proof. Finally,  we use  our  LDP  probability  results  to  prove  our  main  Theorem.

\section{Main Result}\label{section2}\label{main}
\subsection{The Model}\label{model}
Let $[n]=\{1,2,3,...,n\}$   be  fixed set of $n$  sites and  $\skrig(n,d)$  be  the  set  of  all  uniformly   random $d-regular$  graphs  no sites $[n]$ and  fixed   degree $d.$   We  note  that the  edges for  the  configuration  model  via  the  half-edges is  given  by  $E\subset\skrie:=\big\{1,2,3,...nd\big\}$   and  $|E|\in[0,\,n(n-1)/2].$  Let  $[q]=\{1,2,3,...q\}$  denote  spin  set.  We  define  a  function $\eta$  by  $\eta:[n]\to[q].$

For a probability distribution $\mu$ on the  spin set $[q]$   we  define  the Potts  model on  uniformly {\em $d-$regular spinned random
graph} or\emph{ spinned  graph}~$G_d([n],[q])$ with $n$ sites may  be  obtained
as follows:

 Sample  from $\skrig(n,d)$ a  $d-$regular random  graph. we assign  to sites  $[n]$ a spin configuration  $\eta$  according  to  the  spin  law  $\mu^n$  and  note  that  $\skriv_1,\skriv_2,\skriv_3,....,\skriv_q$  partition  the  sites set $[n]$  in such  away that $card(\skriv_i)=n\mu(i)$  for  all  $i\in[q].$  Always observe  $G_d([n],[q])=\{\big(\eta(i)\,:\,i\in [n]\big),E\}$ under
the joint probability measure of graph and spin. We shall call $G_d([n],[q])$  the $d-$regular spinned random
graph model  and interpret $\eta(i)$ as the spin of the site $i$.\\


For each spin  configuration $\eta$, we define  a probability measure, the \emph{empirical spin
measure}~$L_{1}\in\skril([q])$,~by
$$L_{1}(y):=\frac{1}{n}\sum_{j\in E}\delta_{\eta(j))}(y),\quad\mbox{ for $y\in [q]$, }$$
and a symmetric finite measure, the \emph{empirical bond
measure} $L_{2}\in\tilde\skril_*([q]\times[q]),$ by
$$L_{2}(y,x):=\frac{1}{nd}\sum_{(i,j)\in E}[\delta_{(\eta(i),\,\eta(j))}+
\delta_{(\eta(j),\,\eta(i))}](y,x),\quad\mbox{ for
$x,y\in[q]$. }$$ The total mass $\|L_{2}\|=2|E|/nd$.  Let  $\rho$  be the probability distribution on $[q]$ and  $\nu$  be the probability distributions on $[q]^2$ such that $\nu_{ij}=\nu_{ji}$ for all $i,j\in[q].$ Then $(\rho,\nu)$ is said to be admissible for  the graph $G$  on  $n$ vertices and $nd/2$ half-edges if  $\rho=\nu_1.$ i.e. the  $[q]$- marginal of  $\nu.$\\

And, for any finite or countable set $\skriy$ we denote by
$\skril(\skriy)$ the space of probability distributions, and by
$\tilde\skril(\skriy)$ the space of finite distributions on $\skriy$,
both endowed with the topology   generated   by  the total variation
norm.  Further, if $\omega\in\tilde\skril(\skriy)$ and
$\nu_1\ll\omega$ we denote by
$$H(\nu_1\,\|\,\omega)=\big\langle\nu_1,\,
\log\sfrac{\nu_1}{\omega}\big \rangle$$ the \emph{relative
entropy} of $\nu_1$ with respect to $\omega$. We set
$H(\nu_1\,\|\,\omega)=\infty$ if
$\nu_1\not\ll\omega$. 
Finally, we denote by $\tilde\skril_*(\skriy \times \skriy)$ the
subspace of symmetric distributions in $\tilde\skril(\skriy \times
\skriy)$.\\

\subsection{Main Results}\label{main}


\begin{theorem}\label{Ising.1}\label{main1}\label{mainpart}
Suppose that $G([n],d)$  is a  Potts  model  on  $d-$ regular  graph with with  spin  law  $\mu:[q]\to [0,\,1].$   Then, the pair  $(L_1,\,L_2)$ obeys  an  LDP  on  $\skril([q])\times\skril_*([q]\times[q])$   with  speed  $n$  and  convex,  good  rate  function

$$\begin{aligned}
I(\rho,\nu)= \left\{ \begin{array}{ll}H(\rho\, |\, \mu)+\sfrac{d}{2}H(\nu\,|\, \rho\otimes \rho)&\,\,\mbox{ if  $(\rho,\nu)$  is  admissible  }\\
\infty & \mbox{otherwise.}
\end{array}\right.
\end{aligned}$$
\end{theorem}

\section{Proof of  Large Deviation Theorem~\ref{mainpart}}\label{extension}

\subsection{Large  Deviation Probabilities.}\label{LDP}
Now  let   $(\rho,\nu)$  be  admissible  for the  graph  $G_d([n],[q]).$  
Then  we  have
\begin{equation}\label{LDPP}
 P\Big((L_1, L_2)=(\rho,\nu)\Big)=\prod_{i=1}^{n}\mu(\eta(i))\Big(\heap{n}{n\rho_1,n\rho_2,...,n\rho_q}\Big )\prod_{i=1}^{q}\Big(\heap{nd\rho_i}{nd\nu_{i1},nd\nu_{i2},...nd\nu_{iq}}\Big)
\prod_{i<j}^{q}\sqrt{(nd\nu_{ij})!}\sfrac{\prod_{i=j}(nd\nu_{ij}-1)!!}{(nd-1)!!},
\end{equation}
where  $\log (n-1)!!:=\sfrac{n}{2}\log n-\sfrac{n}{2}+o(\log n)$  and $\sum_{j=1}^{q}\nu_{ij}=\rho_i, \, i=1,2,3,...,q.$

\begin{lemma}\label{LDPP1}
Let  $G_d([n],[q])$  be  a  uniformly  random  regular  graphs  with  spin  law  $\mu.$  Let  $(\rho,\nu)$  be  admissible  for  the  graph $G_d([n],[q]).$  Then,  the  pair  $(L_1,\,L_2)$   satisfies

 \begin{equation}
 e^{-nH(\rho\, |\, \mu)-\sfrac{nd}{2}H(\nu\,|\, \rho\otimes \rho)-o(\log n)}\le P\Big((L_1, L_2)=(\rho,\nu)\Big)
 \le e^{-nH(\rho\, |\, \mu)-\sfrac{nd}{2}H(\nu\,|\, \rho\otimes \rho)+o(\log n)}.
 \end{equation}

\end{lemma}
\begin{proof}
Using  the  refined  Stirling's  approximation  $\log n!=n\log n +o(\log n),$  see (Feller, 1971),and  similar arguments as, Dembo et. al(2014, pp. 10)  we  obtain  the  estimate  of  the  large  deviation  probabilities  in  equation~\ref{LDPP}  as  follows:

 $$\prod_{i=1}^{q}\Big(\heap{nd\rho_i}{nd\nu_{i1},nd\nu_{i2},...nd\nu_{iq}}\Big)=e^{nd\langle \log \rho,\,\rho\rangle-nd\langle \log \nu,\, \nu\rangle}=e^{\langle \log \rho,\,\nu\rangle-nd\langle \log \nu,\, \nu\rangle+o(\log n)},$$
  $$\prod_{i<j}^{q}\sqrt{(nd\nu_{ij})!}\prod_{i=j}(nd\nu_{ij}-1)!!=e^{-\sfrac{nd}{2}+\sfrac{nd}{2}\langle \log \nu,\,\nu\rangle+\sfrac{nd}{2} \log nd+o(\log n) }$$  and  $$(nd-1)!!=e^{\sfrac{nd}{2}\log nd-\sfrac{nd}{2}+o(\log n)}.$$
 Moreover, $$\prod_{i=1}^{n}\mu(\eta(i))\Big(\heap{n}{n\rho_1,n\rho_2,...,n\rho_q}\Big )=e^{n\langle \rho,\,\log \mu\rangle-n\langle \rho,\,\log \rho\rangle+o(\log n)}
 $$
Therefore,  putting  all  estimates  together  we  have

 \begin{equation}
 \begin{aligned}\label{LDp}
 &e^{\sfrac{nd}{2}\langle \log \rho,\,\nu\rangle+\sfrac{nd}{2}\langle \log \rho,\,\nu\rangle-\sfrac{nd}{2}\langle \log \nu,\, \nu\rangle+n\langle \rho,\,\log \mu\rangle-n\langle \rho,\,\log \rho\rangle-o(\log n)}\\
 &\le P\Big((L_1, L_2)=(\rho,\nu)\Big)\le e^{\sfrac{nd}{2}\langle \log \rho,\,\nu\rangle+\sfrac{nd}{2}\langle \log \rho,\,\nu\rangle-\sfrac{nd}{2}\langle \log \nu,\, \nu\rangle+n\langle \rho,\,\log \mu\rangle-n\langle \rho,\,\log \rho\rangle+o(\log n)}.
 \end{aligned}
 \end{equation}
 which  ends  the  proof  of  the  Lemma
\end{proof}

\subsection{Method  of  Types  for  the  empirical  measures  of  $d-$  regular random graphs}
Let  $$\skrik_n:=\Big\{ (\omega,\nu)\in \skril(\Gamma)\times\skril(\Gamma\times\Gamma):(L_1,\,L_2)=(\omega,\nu)\Big\}$$ i.e  $\skrik_n$  is  the  set  of  all possible  types  of  $(L_1,\,L_2)$  of  length  $n$  in $ \skril(\Gamma)\times\skril(\Gamma\times\Gamma).$  Then,  observe  that  we  have
 $$|\skrik_n|\le (n+1)^{|\Gamma|(1+|\Gamma|)}.$$
 Let  $F$  be subset  of $\skril(\Gamma)\times\skril(\Gamma\times\Gamma)$   and  observe  that $$P\Big ((L_1,L_2)\in F\Big)=\sum_{(\rho,\, \nu)\in F\cup \skrik_n}P\Big((L_1,L_2)=(\rho,\, \nu)\Big).$$

 Using  Lemma~\ref{LDPP1}   we  have  $$ \begin{aligned}
   e^{-n\inf_{(\rho,\, \nu)\in F^o\cap \skrik_n}I(\rho,\,\nu)+o(\log n)}&\le \sup_{(\rho,\, \nu)\in F^o\cap \skrik_n}e^{-nI(\rho,\,\nu)-o(\log n) }\le |\skrik_n|^{-1}\sup_{(\rho,\, \nu)\in F^o\cap \skrik_n}P\Big((L_1,L_2)=(\rho,\, \nu)\Big)\\
 &\le P\Big ((L_1,L_2)\in F\Big)\le |\skrik_n|\sup_{(\rho,\, \nu)\in F^c\cap \skrik_n}P\Big((L_1,L_2)=(\rho,\, \nu)\Big)\\
 &\le \sup_{(\rho,\, \nu)\in F^c\cap \skrik_n}e^{-nI(\rho,\,\nu)+o(\log n)}\le   e^{-n\inf_{(\rho,\, \nu)\in F^c\cap \skrik_n}I(\rho,\,\nu)+o(\log n)} \end{aligned}$$.

 This  gives

 $$ \begin{aligned}
 -\limsup_{n\to\infty}\inf_{(\rho,\, \nu)\in F^o\cap \skrik_n}I(\rho,\,\nu)&\le \liminf_{n\to\infty}\sfrac{1}{n}\log P\Big ((L_1,L_2)\in F\Big)\\
 &\le \limsup_{n\to\infty}\sfrac{1}{n}\log P\Big ((L_1,L_2)\in F\Big)\le  - \liminf_{n\to\infty}\sfrac{1}{n}\inf_{(\rho,\, \nu)\in F^c\cap \skrik_n}I(\rho,\,\nu) \end{aligned}$$.

Now  we  note  that  $F^c\cap\skrik_n\subset F^c$  for  all  $n$  and  so  we  have $$\liminf_{n\to\infty}\sfrac{1}{n}\inf_{(\rho,\, \nu)\in F^c\cap \skrik_n}I(\rho,\,\nu)\le\inf_{(\rho,\, \nu)\in F^c}I(\rho,\,\nu)$$  which  establishes  the  upper  bound  of  the  main  results.  Using  Similar  arguments  as  in  \cite[page~17]{DZ98}  we  obtain  $$\limsup_{n\to\infty}\inf_{(\rho,\, \nu)\in F^o\cap \skrik_n}I(\rho,\nu)\le\inf_{(\rho,\, \nu)\in F^o\cap \skrik_n}I(\rho,\,\nu)$$ which  completes  the  proof  of  Theorem~\ref{main}.
\section{Conclusion}
In  this  paper  we  have  presented  a  joint  LDP  for  the  empirical  spin  measure  and  the  empirical  co-operate  measure  of d-regular  random  graphs.The  main  technique  used  to  establish  this  result  is  combinatoric  arguments,  see  Dembo et. al  (2014)   and  the  method  of  types,  see Dembo and Zeitouni (1998, Pages~16-17). This  LDP  result  could  be  used  to  prove  a  Lossy  Asymptotic  Equipartion  Property  for  the  d-regular  random  graphs   and  to establish   annealed  asymptotic  result  about  the  random partition  function  of  the  d-regular  random  graphs .



\bigskip

\begin{thebibliography}{WWW98}


\bibitem[1]{DZ98}
{\sc A.~Dembo} and {\sc O.~Zeitouni.}
\newblock Large deviations techniques and applications.
\newblock Springer, New York, (1998).
\smallskip

\bibitem[2]{DA06}
{\sc K.~Doku-Amponsah.}
\newblock{Large deviations and basic information theory for hierarchical and networked data structures.}
\newblock PhD Thesis, Bath (2006).
\smallskip

\bibitem[3]{Fel71}
{\sc W.~Feller.}
\newblock{An introduction to Probability Theory and Its Applications,}
\newblock{Vol.~I, Wiley, New York, second edition, 1971.}
\smallskip


\bibitem[4]{5}
{\sc A.~Dembo}, {\sc A.~Montanari}, {A.~Sly}  and  N.~ Sun
\newblock{The  replica symmetric  solution  for  Potts  models  on  d-regular  graphs.}
\newblock{ \emph{Comm. Math.  Phy. 327:2(2014),pp. 551-575.}}

\smallskip



\bibitem[5]{7}
{\sc S.~Dommers}, {\sc C.~Giardina.}  and  {\sc R.V.D.~Hofstad}
\newblock{Ising  models  on power-law  radom  graphs.}
\newblock {\emph{J. Stat. Phys. (2010) 141: 638–660
DOI 10.1007/s10955-010-0067-9}}
\smallskip


\bibitem[6]{9}
{\sc S. Dommers },{\sc C.~Giardina}, {\sc C.~Giberti}  and {\sc R.~ van der Hofstad}
\newblock{Ising critical  exponents  on  random  trees  and  graphs.}
\newblock {\emph{Comm. in  Math.  Physics, 328(1):335-395(2014)}.}
\smallskip


\bibitem[7]{10}
{\sc K.~Doku-Amponsah}  and {\sc P.~M\"orters.}
\newblock{Large deviation principle for  empirical measures of
coloured random graphs.}
\newblock {\emph{Ann. Appl. Prob. Volume 20, Number 6 (2010), 1989-2021.}}


\bibitem[8]{11}
{\sc K.~Doku-Amponsah.}
\newblock{Asymptotic equipartition properties for hierarchical and networked  structures.}
\newblock ESAIM:\emph{ Probability  and Statistics}.DOI: 10.1051/ps/2010016 :
Published online by Cambridge University Press: 03 February 2011.
\smallskip

\smallskip


\bibitem[9]{12}
{\sc S.~Dommers.}
\newblock {Spin  models  on  random  graphs.}
\newblock {PhD Thesis, Eindhoven(2013).}
\smallskip


\end{thebibliography}
\end{document}